\documentclass[12pt]{amsart}
\usepackage{etex}
\usepackage{amsfonts, amssymb, latexsym}

\usepackage{amscd,amssymb, amsmath}
\usepackage{verbatim}
\usepackage{pstricks}
\usepackage{pst-node}
\usepackage[mathscr]{euscript}
\usepackage{tikz}
\usetikzlibrary{matrix,arrows}
\usepackage[normalem]{ulem}
\usepackage{varwidth}
\usepackage{leftidx}
\usepackage[most]{tcolorbox}
\usepackage[utf8]{inputenc}
\usepackage{ulem}
\usepackage{expl3}

\newcommand\numberthis{\addtocounter{equation}{1}\tag{\theequation}}

\newsymbol\pp 1275
\usepackage{tikz}
\usetikzlibrary{matrix,arrows,backgrounds,shapes.misc,shapes.geometric,patterns,calc,positioning}

\usepackage[colorlinks=true, pdfstartview=FitV, linkcolor=blue, citecolor=blue, urlcolor=blue]{hyperref}

\usepackage{fullpage}
\usepackage{graphicx}
\usepackage{enumerate}
\usepackage{multirow,bigdelim}

\def\VR{\kern-\arraycolsep\strut\vrule &\kern-\arraycolsep}
\def\vr{\kern-\arraycolsep & \kern-\arraycolsep}

\newcounter{algGcount}
\newcounter{algPcount}
 \newcommand{\verteq}{\rotatebox{90}{$\,=$}}

\newtcbtheorem[use counter=algGcount, number format=\Alph]{alg}{Algorithm}{}{}
\newtcbtheorem[use counter=algPcount, number format=\Alph]{algP}{Algorithm}{}{}

\newtheorem{theorem}{Theorem}

\newtheorem{lemma}[theorem]{Lemma}
\newtheorem{prop}[theorem]{Proposition}
\newtheorem{corollary}[theorem]{Corollary}

\theoremstyle{theorem}

\theoremstyle{definition}
\newtheorem{definition}[theorem]{Definition}

\newtheorem*{prob}{Problem}

\newtheorem{rmk}[theorem]{Remark}

\newenvironment{remark}[1][]{\begin{rmk}[#1]\pushQED{\qed}}{\popQED \end{rmk}}

\newtheorem{qu}[theorem]{Question}

\newtheorem*{rmknonum}{Remark}

\newtheorem{obs}[theorem]{Observation}

\newtheorem{ex}[theorem]{Example}
\newenvironment{example}[1][]{\begin{ex}[#1]\pushQED{\qed}}{\popQED \end{ex}}

\newcommand{\rep}{\operatorname{rep}}

\newcommand{\ZZ}{\mathbb Z}
\newcommand{\CC}{\mathbb C}
\newcommand{\RR}{\mathbb R}
\newcommand{\NN}{\mathbb N}
\newcommand{\QQ}{\mathbb Q}

\newcommand{\A}{\mathcal  A}
\newcommand{\X}{\mathcal X}
\newcommand{\vv}{\mathcal V}
\newcommand{\dsp}{\displaystyle} 
\newcommand{\I}{\mathcal I}

\newcommand{\crk}{\operatorname{corank}}
\newcommand{\disc}{\operatorname{disc}}
\newcommand{\tr}{\operatorname{tr}}

\newcommand{\V}{\ensuremath{V_{\mathcal{X}}}}

\newcommand{\Ima}{\operatorname{Im}}

\newcommand{\ddim}{\operatorname{\mathbf{dim}}}

\newcommand{\dd}{\operatorname{\mathbf{d}}}

\newcommand{\ff}{\operatorname{\mathbf{f}}}

\newcommand{\K}{\mathcal{K}}

\newcommand{\ar}{\mathcal{A}}
\newcommand{\s}{\mathcal{S}}
\newcommand{\jj}{\mathcal{I}^{-}_j}
\newcommand{\ii}{\mathcal{I}^{+}_i}
\newcommand{\Span}{\mathsf{Span}}
\newcommand{\capa}{\mathbf{cap}}

\newcommand{\rk}{\operatorname{rank}}

\DeclareMathOperator{\cp}{cap}
\newcommand\restr[2]{{
  \left.\kern-\nulldelimiterspace 
  #1 
  \vphantom{\big|} 
  \right|_{#2} 
  }}

\begin{document}
\title{Simultaneous robust subspace recovery and semi-stability of quiver representations}
\author{Calin Chindris and Daniel Kline}
\address{University of Missouri-Columbia, Mathematics Department, Columbia, MO, USA}
\email[Calin Chindris]{chindrisc@missouri.edu}

\address{College of the Ozarks, Mathematics Department, Point Lookout, MO, USA}
\email[Daniel Kline]{dkline@cofo.edu}

\date{\today}
\bibliographystyle{amsalpha}
\subjclass[2010]{16G20, 13A50, 14L24}
\keywords{Capacity of completely positive operators, partitioned data sets, semi-stability of quiver representations, shrunk subspaces, simultaneous robust subspace recovery, Wong sequences}

\begin{abstract} We consider the problem of \emph{simultaneously} finding lower-dimensional subspace structures in a given $m$-tuple $(\X^1, \ldots, \X^m)$ of possibly corrupted, high-dimensional data sets all of the same size. We refer to this problem as \emph{simultaneous robust subspace recovery} (SRSR) and provide a quiver invariant theoretic approach to it. We show that SRSR is a particular case of the more general problem of effectively deciding whether a quiver representation is semi-stable (in the sense of Geometric Invariant Theory) and, in case it is not, finding a subrepresentation certifying in an optimal way that the representation is not semi-stable. In this paper, we show that SRSR and the more general quiver semi-stability problem can be solved effectively.
\end{abstract}

\maketitle
\setcounter{tocdepth}{1}
\tableofcontents

\smallskip

\section{Introduction}
\subsection{Formulation of the \textbf{SRSR} problem} Given a point cloud, classical robust subspace recovery asks for a lower-dimensional subspace containing a sufficiently large number of data points. In this paper, we take a quiver invariant theoretic approach to this problem. This allows us to work with multiple point clouds, all of the same cardinality, \emph{simultaneously}. To decide what it means for a lower-dimensional subspace structure to contain enough points across a collection of point clouds, we use the Hilbert-Mumford numerical criterion for semi-stability of quiver representations \cite{K} as our guiding principle.

Let $D, m ,d_1, \ldots, d_m$ be positive integers such that $D=\sum_{j=1}^m d_j$. Let $\X=\{X_1, \ldots, X_n\} \subseteq \RR^D$ be a data set such that the first $d_1$ coordinates of each one of the vectors $X_i$ encode measurements of a certain type-$1$, the next $d_2$ coordinates encode measurements of type-$2$, and so on, with the last $d_m$ coordinates encoding measurements of type-$m$. We write  
\[
\begin{array}{ccccccc} 
 															&	 X_1,	 &	\cdots  & X_i,  & 	\cdots	&       X_n &  \\
															& \verteq &	 & \verteq &   &  \verteq &  \\
\begin{matrix} \X^j \subseteq \RR^{d_j} \left.   \rule{0in}{.1in} \right \{  \end{matrix}  & \left [ \begin{matrix} v_1^1 \\ \vdots \\ v_1^j \\ \vdots \\ v_1^m \\ \end{matrix} \right ]  & & \left [ \begin{matrix} v_i^1 \\ \vdots \\ v_i^j \\ \vdots \\ v_i^m \\ \end{matrix} \right ] & & \left [ \begin{matrix} v_n^1 \\ \vdots \\ v_n^j \\ \vdots \\ v_n^m \\ \end{matrix} \right ] & \} \mbox{ type-$j$ measurements}   \\
\end{array}
\]
with $v^j_i \in \RR^{d_j}, i \in [n], j \in [m]$, and refer to $\X$ as a \emph{partitioned data set in $\RR^D$ of type} $(n,m,\underline{d})$ where $\underline{d}=(d_1, \ldots, d_m)$. It gives rise to the $m$-tuple of data sets $(\X^1, \ldots, \X^m)$ where $\X^j:=\{v^j_1, \ldots, v^j_n\} \subseteq \RR^{d_j}$, $j \in [m]$.  Conversely, any tuple of data sets, all of them of the same size, defines a partitioned data set.

Let us now consider the $m$-tuple of data sets $\X^j=\{v^j_1, \ldots, v^j_n\} \subseteq \RR^{d_j}$, $j \in [m]$. Our goal in this paper is to perform \emph{robust subspace recovery} on the data sets $\X^1, \ldots, \X^m$ \emph{simultaneously}. To explain what this means, let us consider an $m$-tuple of subspaces $(T_1, \ldots, T_m)$ with $T_j \subseteq \RR^{d_j}$, $j \in [m]$, and set
\[
\I_T:=\{i \in [n] \mid v^j_i \in T_j, \forall j \in [m]\}. 
\]
Thus the initial data set $\X \subseteq \RR^D$ contains precisely $|\I_T|$ vectors such that for each one of these vectors the first $d_1$ coordinates are from $T_1$, the next $d_2$ coordinates are from $T_2$, and so on. We are now ready to state what we mean by \emph{simultaneous robust subspace recovery} (SRSR). 

\begin{prob} \label{main-prob-1} Given an $m$-tuple of data sets $\X^j:=\{v^j_1, \ldots, v^j_n\} \subseteq \RR^{d_j}$, $j \in [m]$, is there an effective way to find, whenever possible, an $m$-tuple of subspaces $(T_1, \ldots, T_m)$ with $T_j \subseteq \RR^{d_j}$, $j \in [m]$, such that
\begin{equation} \label{main-ineq}
|\I_T|> {\sum_{j=1}^m \dim T_j \over D} \cdot n?
\end{equation}
By an \emph{SRSR solution for $\X$}, we mean any $m$-tuple of subspaces $(T_1, \ldots, T_m)$ for which $(\ref{main-ineq})$ holds. 
\end{prob}

When $m=1$, this is a central problem in \emph{Robust Subspace Recovery} (RSR)  (see for example \cite{Har-Moi-2012}, \cite{LerMau-2018}).  In \cite{Har-Moi-2012}, the authors argue that the RSR version ($m=1$) of (\ref{main-ineq}) is the optimal balance between robustness and efficiency. For arbitrary $m \geq 1$, we point out that $(\ref{main-ineq})$ is intimately related to the Hilbert-Mumford numerical criterion for semi-stability applied to quiver representations (see Section \ref{background-sec} for more details).  We point out that in general, SRSR cannot be achieved by simply using the classical RSR (see Examples \ref{srsr-ex1}-\ref{srsr-ex4}).

\subsection{The quiver approach}\label{quiver-approach-intro-sec} Our approach to SRSR is based on quiver invariant theory (see Section \ref{background-sec} for more details). We begin by viewing the $m$-tuple of data sets $\X^j \subseteq \RR^{d_j}$, $j \in [m]$, as a representation $\V$ of the complete bipartite quiver $\K_{n,m}$ with $n$ source vertices and $m$ sink vertices as follows:
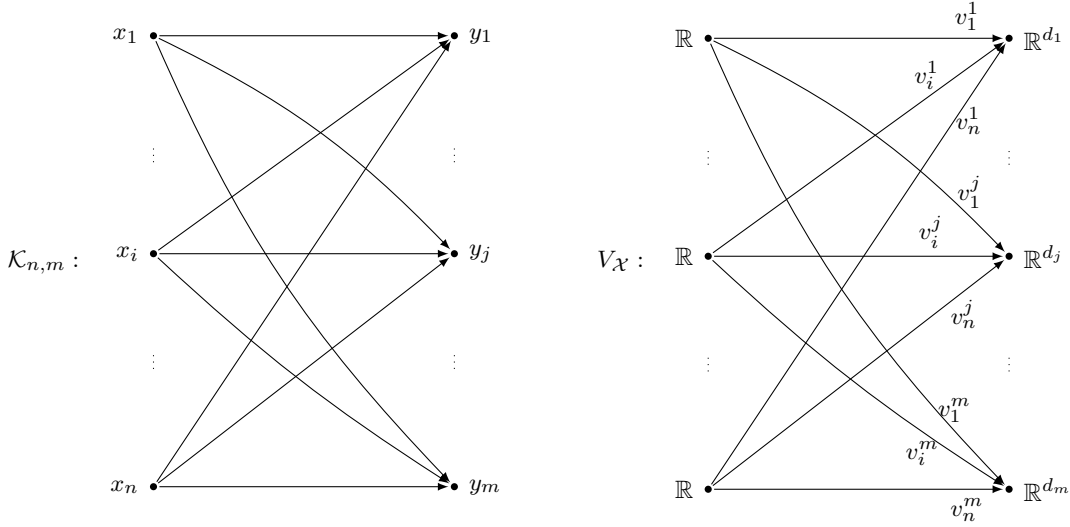
\begingroup\makeatletter\def\f@size{9.5}\check@mathfonts
\[ \K_{n,m}:~
\vcenter{\hbox{  
\begin{tikzpicture}[point/.style={shape=circle, fill=black, scale=.3pt,outer sep=3pt},>=latex]
   \node[point,label={left:$x_1$}] (1) at (-4,3) {};
   \node[point,label={left:$x_i$}] (2) at (-4,.1) {};
   \node[point,label={left:$x_n$}] (3) at (-4,-3) {};
   
   \node[point,label={right:$y_1$}] (-1) at (0,3) {};
   \node[point,label={right:$y_j$}] (-2) at (0,.1) {};
   \node[point,label={right:$y_m$}] (-3) at (0,-3) {};
  
   \draw[dotted] (0,1.5)--(0,1.25);
   \draw[dotted] (-4,1.5)--(-4,1.25);
   \draw[dotted] (0,-1.25)--(0,-1.5);
   \draw[dotted] (-4,-1.25)--(-4,-1.5);
  
   \path[->]
  (1) edge  (-1)
  (1) edge  [bend left=10] (-2)
  (1) edge  [bend right=10] (-3)
  (2) edge  (-1)
  (2) edge (-2)
  (2) edge [bend right=5]  (-3)
  (3) edge  (-1)
  (3) edge (-2)
  (3) edge  (-3);
\end{tikzpicture} 
}}
\hspace{30pt}
\V:~
\vcenter{\hbox{
\begin{tikzpicture}[point/.style={shape=circle, fill=black, scale=.3pt,outer sep=3pt},>=latex]
   \node[point,label={left:$\RR$}] (1) at (-4,3) {};
   \node[point,label={left:$\RR$}] (2) at (-4,.1) {};
   \node[point,label={left:$\RR$}] (3) at (-4,-3) {};
   \node[point,label={right:$\RR^{d_1}$}] (-1) at (0,3) {};
   \node[point,label={right:$\RR^{d_j}$}] (-2) at (0,.1) {};
   \node[point,label={right:$\RR^{d_m}$}] (-3) at (0,-3) {};
   
   \draw[dotted] (0,1.5)--(0,1.25);
   \draw[dotted] (-4,1.5)--(-4,1.25);
   \draw[dotted] (0,-1.25)--(0,-1.5);
   \draw[dotted] (-4,-1.25)--(-4,-1.5);
  
   \path[->]
   (1) edge node[very near end, above] {$v^1_1$} (-1)
   (1) edge [bend left=10] node[very near end, above] {$v^j_1$} (-2)
   (1) edge [bend right=10] node[very near end, above] {$v^m_1$} (-3)
   (2) edge [bend left=0] node[near end, above] {$v^1_i$} (-1)
   (2) edge node[near end, above] {$v^j_i$} (-2)
   (2) edge [bend right=5] node[near end, below] {$v^m_i$} (-3)
   (3) edge node[very near end, below] {$v^1_n$} (-1)
   (3) edge [bend left=0] node[very near end, below] {$v^j_n$} (-2)
   (3) edge node[very near end, below] {$v^m_n$} (-3);  
\end{tikzpicture} 
}}
\]\endgroup

The dimension vector of $\V$, denoted by $\ddim \V \in \NN^{nm}$, simply records the dimensions of the vector spaces attached to the vertices of $\mathcal{K}_{n,m}$, i.e. 
$$\ddim\V(x_i)=1, \forall i \in [n], \text{~and~} \ddim \V(y_j)=d_j, \forall j \in [m].$$ 
By a \emph{weight} of a quiver, we simply mean an assignment of integer numbers to the vertices of the quiver. In the SRSR setup, we are primarily interested in the weight $\sigma_0 \in \ZZ^{nm}$ defined by
\begin{equation} \label{signot}
\sigma_0(x_i)=D, \forall i \in [n], \text{~and~} \sigma(y_j)=-n, \forall j \in [m].
\end{equation} 

One of the key concepts in quiver invariant theory is that of a semi-stable quiver representation with respect to a weight (see Definition \ref{semi-stab-defn}). It turns out that the partitioned data set $\X=(\X^1, \ldots, \X^m)$, viewed as the representation $\V$ of $\K_{n,m}$, is $\sigma_0$-semi-stable if and only if 
\begin{equation}\label{data-semi-stab-ineq}
D\cdot |\I_T|-n\cdot \left(\sum_{j=1}^m \dim T_j \right) \leq 0,
\end{equation}
for every $m$-tuple of subspaces $(T_1, \ldots, T_m)$ with $T_j \subseteq \RR^{d_j}$, $j \in [m]$. (For more details, see Remark \ref{SRSR-gen-ex}.) Thus Problem \ref{main-prob-1} can be rephrased as asking to decide whether the representation $\V$ is \emph{not} $\sigma_0$-semi-stable and, if that is the case, to find a subrepresentation of $\V$ for which $(\ref{data-semi-stab-ineq})$ does \emph{not} hold. 

Our quiver approach allows us to establish a bridge between quiver semi-stability, on one side, and key concepts in algebraic complexity,  such as the capacity and shrunk subspaces associated to completely positive operators, on the other.  This approach combined with known algorithms for dealing with the aforementioned concepts from algebraic complexity,  due to L. Gurvits \cite{Gurv04}, and G. Ivanyos, Y. Qiao, and K.V. Subrahmanyam \cite{IQS18},  yields the following results.

\begin{theorem} \label{main-thm} Let  $\X^j:=\{v^j_1, \ldots, v^j_n\} \subseteq \RR^{d_j}$, $j \in [m]$, be an $m$-tuple of data sets. Then there exist a deterministic polynomial time algorithm and a probabilistic algorithm to decide whether $(\X^1, \ldots, \X^m)$ has an SRSR solution and to compute such a solution in case it exists. 

Moreover, if $(\X^1, \ldots, \X^m)$ admits an SRSR solution then any output $(T_1, \ldots, T_m)$ of the deterministic or probabilistic algorithm is an \emph{optimal} SRSR solution  in the following sense.  If $(T'_1, \ldots, T'_m)$ is any SRSR solution for $(\X^1, \ldots, \X^m)$ then
\begin{equation} \label{optimal} 
|\I_T|- {\sum_{j=1}^m \dim T_j \over D}\cdot n \geq |\I_{T'}|- {\sum_{j=1}^m \dim T'_j \over D}\cdot n.
\end{equation}
\end{theorem}

The probabilistic algorithm (Algorithm \ref{algPlab}) in the theorem above is a randomized version of the deterministic polynomial time algorithm from \cite{IQS18}. We point out that Algorithm \ref{algPlab}, based on the Schwartz-Zippel Lemma \cite{Schwartz-lemma-1980, Zippel-lemma-1979}, is easier to implement in practice.

In the general context of quiver representations, the \emph{discrepancy} of a quiver datum $(V, \sigma)$, where $V$ is a representation and $\sigma$ is a non-zero weight of a quiver, is defined as
$$
\disc(V, \sigma)=\max\{\sigma \cdot \ddim W \mid W \text{~is a subrepresentation of~} V\}.
$$
(We refer to Section \ref{general-semi-stab-sec} for the details behind our notations.) We have the following result, extending Theorem \ref{main-thm} to more general quiver representations.

\begin{theorem}\label{main-thm-2} Let $Q$ be a quiver without oriented cycles and let $(V, \sigma)$ be a quiver datum (defined over $\QQ$) such that $\sigma \cdot \ddim V=0$. 

\begin{enumerate}[(i)]
\item Then there exist deterministic polynomial time algorithms to check whether $V$ is $\sigma$-semi-stable or not.

\item Assume that $Q$ is a bipartite quiver (not necessarily complete), and that $\sigma$ is positive at the source vertices and negative at the sink vertices. Then there exists a deterministic polynomial time algorithm that constructs a subrepresentation $W \subseteq V$ such that 
$$
\disc(V, \sigma)=\sigma \cdot \ddim W.
$$
In particular, if $\sigma\cdot \ddim W=0$ then $V$ is $\sigma$-semi-stable. Otherwise, $W$ is an optimal witness to $V$ not being $\sigma$-semi-stable.
\end{enumerate}
\end{theorem}

\subsection*{Acknowledgment} The authors would like to thank to Petros Valettas for many enlightening discussions on the paper, especially for suggesting the use of the Schwartz-Zippel Lemma in Proposition \ref{prop-prob-alg}.  The authors are also thankful to the anonymous referee for comments improving  the paper, and for raising the interesting question about a possible connection between our results and maximally destabilizing 1-parameter subgroups in GIT.  Thanks are also due to Chi-yu Cheng  for asking the same question and for other interesting comments on the paper.  

C. Chindris is supported by Simons Foundation grant $\# 711639$.

\section{Background on quiver invariant theory}\label{background-sec}
Throughout, we work over the field $\RR$ of real numbers and denote by $\NN=\{0,1,\dots \}$. For a positive integer $L$, we denote by $[L]=\{1, \ldots, L\}$.

\subsection{Semi-stability of quiver representations} A \emph{quiver} $Q=(Q_0,Q_1,t,h)$ consists of two finite sets $Q_0$ (\emph{vertices}) and $Q_1$ (\emph{arrows}) together with two maps $t:Q_1 \to Q_0$ (\emph{tail}) and $h:Q_1 \to Q_0$ (\emph{head}). We represent $Q$ as a directed graph with set of vertices $Q_0$ and directed edges $a:ta \to ha$ for every $a \in Q_1$. 

A \emph{representation} of $Q$ is a family $V=(V(x), V(a))_{x \in Q_0, a\in Q_1}$ where $V(x)$ is a finite-dimensional $\RR$-vector space for every $x \in Q_0$, and $V(a): V(ta) \to V(ha)$ is an $\RR$-linear map for every $a \in Q_1$. A \emph{subrepresentation} $W$ of $V$, written as $W \subseteq V$, is a representation of $Q$ such that $W(x) \subseteq V(x)$ for every $x \in Q_0$, and $V(a)(W(ta)) \subseteq W(ha)$ and $W(a)$ is the restriction of $V(a)$ to $W(ta)$ for every arrow $a \in Q_1$. Throughout, we assume that all of our quivers have no oriented cycles.

The dimension vector $\ddim V \in \NN^{Q_0}$ of a representation $V$  is defined by $\ddim V(x)=\dim_{\RR} V(x)$ for all $x \in Q_0$. By a dimension vector of $Q$, we simply mean a $\NN$-valued function on the set of vertices $Q_0$. For two vectors $\sigma, \dd \in \RR^{Q_0}$, we denote by $\sigma \cdot \dd$ the scalar product of the two vectors, i.e. $\sigma \cdot \dd=\sum_{x \in Q_0} \sigma(x)\dd(x)$. 

\begin{definition}\label{semi-stab-defn} Let $\sigma \in \ZZ^{Q_0}$ be a weight of $Q$. A representation $V$ of $Q$ is said to be \emph{$\sigma$-semi-stable} if 
\[
\sigma \cdot \ddim V=0 \text{~and~} \sigma \cdot \ddim W \leq 0
\]
for all subrepresentations $W$ of $V$.
\end{definition}
\noindent
We point out that, as shown by King \cite{K}, these linear homogeneous inequalities come from the Hilbert-Mumford's numerical criterion for semi-stability applied to quiver representations.

\begin{remark}[\textbf{Classical RSR from quiver semi-stability}] Let $\mathcal{X} = \{v_1, \ldots, v_n\} \subseteq \RR^{D}$ be a data set of $n$ vectors in $\mathbb{R}^D$. We can view $\X$ as the representation  
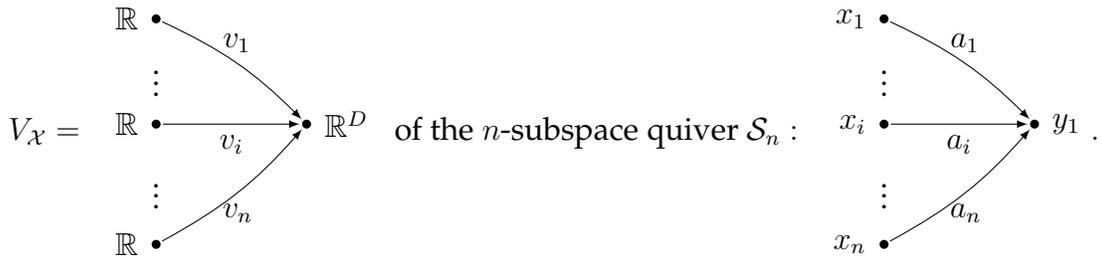
\[
V_{\X} =~\vcenter{\hbox{
\begin{tikzpicture}[point/.style={shape=circle, fill=black, scale=.3pt,outer sep=3pt},>=latex]
   \node[point,label={left:$\RR$}] (1) at (-2,1.5) {};
   \node[point,label={left:$\RR$}] (2) at (-2,.1) {};
   \node[point,label={left:$\RR$}] (3) at (-2,-1.5) {};
   \node[label={center: $\vdots$}] (4) at (-2, .75) {};
   \node[label={center:$\vdots$}] (5) at (-2,-.75) {};
   \node[point,label={right:$\RR^D$}] (-2) at (0,.1) {};

   \path[->]
  (1)  edge [bend left=10] node[above] {$v_1$}  (-2)
  (2) edge node[below] {$v_i$} (-2)
  (3) edge [bend right=10] node[below] {$v_n$}  (-2);

\end{tikzpicture} 
}}
\mbox{ of the $n$-subspace quiver }
\s_n:~\vcenter{\hbox{  
\begin{tikzpicture}[point/.style={shape=circle, fill=black, scale=.3pt,outer sep=3pt},>=latex]
   \node[point,label={left:$x_1$}] (1) at (-2,1.5) {};
   \node[point,label={left:$x_i$}] (2) at (-2,.1) {};
   \node[point,label={left:$x_n$}] (3) at (-2,-1.5) {};
   \node[label={center: $\vdots$}] (4) at (-2, .75) {};
   \node[label={center:$\vdots$}] (5) at (-2,-.75) {};
   \node[point,label={right:$y_1$}] (-2) at (0,.1) {};

   \path[->]
  (1) edge [bend left=10] node[above] {$a_1$} (-2)
  (2) edge node[below] {$a_i$} (-2)
  (3) edge [bend right=10] node[below] {$a_n$}  (-2);

\end{tikzpicture}}}. 
\]
Let us consider the weight $\sigma_0 = (D, \ldots, D, -n)$ and note that $\sigma_0 \cdot \ddim \V = 0$. For a subspace $T \subseteq \RR^D$, let $\I_T = \{ i \in [n]| v_i \in T \}$, and consider the subrepresentation $W_T \subseteq \V$ defined by
\[ 
W_T(x_i) = \begin{cases} \RR & i \in \I_T \\ 0 & i \notin \I_T \end{cases}, W_T(y_1) = T, \mbox{ and } W_T(a_i)=v_i, \forall i \in [n].
\]
It turns out that these are all the subrepresentation of $\V$ that matter when checking whether $\V$ is $\sigma_0$-semi-stable. Indeed it is easy to check that if $W$ is a subrepresentation of $\V$ then
$$
\sigma_0 \cdot \ddim W \leq \sigma_0 \cdot \ddim W_{T},
$$
where $T=W(y_1) \subseteq \RR^D$. Thus we get that $\V$ is $\sigma_0$-semi-stable if and only if
\[ \sigma_0 \cdot \ddim W_T \leq 0, \; \forall T \subseteq \RR^D, \]
which is equivalent to 
$$
|\I_T| \leq \frac{\dim T}{D} \cdot n, \forall T \subseteq \RR^D.
$$
So $\X$ has a lower-dimensional subspace structure precisely when $\V$ is not $\sigma_0$-semi-stable as a representation of $\s_n$. By an \emph{RSR solution for $\X$}, we mean any subspace $T \subseteq \RR^D$ for which 
$$
|\I_T| > \frac{\dim T}{D} \cdot n.
$$
\end{remark}

\begin{remark}[\textbf{SRSR from quiver semi-stability}]\label{SRSR-gen-ex} Let $\X=(\X^1, \ldots, \X^m)$ be a partitioned data set with $\X^j=\{v^j_1, \ldots, v^j_n\} \subseteq \RR^{d_j}$, $j \in [m]$, and let $\V$ be the corresponding representation of the complete bipartite quiver $\K_{n,m}$. Recall that in this case the weight $\sigma_0$ of $\K_{n,m}$ is defined as
$$
\sigma_0(x_i)=D, \forall i \in [n], \text{~and~} \sigma(y_j)=-n, \forall j \in [m],
$$
where $D=\sum_{j \in [m]} d_j$. Every $m$-tuple of subspaces $(T_1, \ldots, T_m)$ with $T_j \subseteq \RR^{d_j}$, $ j \in [m]$, gives rise to a subrepresentation $W_T$ of $\V$ defined by
\[
W_T(x_i)=
\begin{cases}
\RR &\text{if~} i \in \I_T\\
0&\text{if~} i \notin \I_T
\end{cases},
\hspace{8pt}
W_T(y_j)=T_j,
\text{~and~}
W_T(a_{ij})=
\begin{cases}
v^j_i &\text{if~} i \in \I_T\\
0&\text{if~} i \notin \I_T
\end{cases}
\]
where $a_{ij}$ denotes the arrow from $x_i$ to $y_j$ for all $i \in [n]$ and $j \in [m]$. These are all the subrepresentations of $\V$ that matter when checking the $\sigma_0$-semi-stability of $\V$. Indeed, if $W$ is a subrepresentation of $\V$ then
$$
\sigma_0 \cdot \ddim W \leq \sigma_0\cdot W_{T},$$
where $T=(W(y_1), \ldots, W(y_m))$. Thus $\V$ is a $\sigma_0$-semi-stable representation of $\K_{n,m}$ if and only if 
\begin{equation}
\sigma_0 \cdot \ddim W_T=D\cdot |\I_T|-n\cdot \left(\sum_{j=1}^m \dim T_j \right) \leq 0.
\end{equation}
for every $m$-tuple of subspaces $(T_1, \ldots, T_m)$ with $T_j \subseteq \RR^{d_j}$, $j \in [m]$. So $\X$ has a (simultaneous) lower-dimensional subspace structure precisely when $\V$ is not $\sigma_0$-semi-stable as a representation of $\K_{n,m}$. 
\end{remark}

\begin{example} \label{srsr-ex1} This is an example of a data set that has an RSR solution $T$ whose projections do not yield an SRSR solution for the data set viewed as a partitioned data set. Consider the data set $\X \subseteq \RR^5$ given by
\[ \begin{array}{cccc} 
X_1	 & X_2 & X_3 & X_4 \\
\verteq	 & \verteq & \verteq & \verteq	\\ 
 \left [\begin{matrix}  1  \\ 2 \\ 1 \\ 2 \\ 1   \end{matrix} \right ]  &
\left [ \begin{matrix}   5 \\ 4 \\ 3 \\ 6 \\ 3    \end{matrix} \right ]  &
 \left [\begin{matrix}   13 \\ 9 \\ 3 \\ 5 \\ 1    \end{matrix} \right ]  &
\left [\begin{matrix}   3 \\ 1 \\ 2 \\ 2 \\ 9    \end{matrix} \right ] 
\end{array}
 \] 

\vspace{.2in} 
\noindent Then $T=\Span(X_1,X_2)$ is an RSR solution for $\X$ since $ \displaystyle |\I_T| = 2 > \frac{2}{5} \cdot 4$.  Let us now view $\X$ as a partitioned data set of type $(4, 2, (2,3))$. Let $T_1=\Span \left( \left [\begin{matrix}  1  \\ 2   \end{matrix} \right ], \left [\begin{matrix}  5  \\ 4   \end{matrix} \right ] \right)=\RR^2$ and $T_2=\Span \left( \left [\begin{matrix}  1  \\ 2 \\ 1   \end{matrix} \right] \right) \cong \RR$ be the projections of $T$ onto $\RR^2$ and $\RR^3$, respectively. Since $|\I_{(T_1, T_2)}|=2 \leq {2 +1 \over 5}\cdot 4$, the pair of subspaces $(T_1, T_2)$ is not an SRSR solution for $\X$. In fact, it is immediate to check that the partitioned data set $\X$ has no SRSR solutions, i.e. $\V$ is $\sigma_0$-semi-stable as a representation of $\K_{4,2}$.
\end{example}

\begin{example} \label{srsr-ex2} This is an example where each individual data set $\X^j$ admits an RSR solution but the partitioned data set $\X=(\X^1, \ldots, \X^m)$ does not have an SRSR solution. Consider the partitioned data set given by
\[ \begin{array}{ccccc} 
 			&	 X_1	 	 & X_2 		&       X_3 &	 X_4   \\
			& \verteq	 & \verteq	& \verteq  & \verteq  \\ 
\multirow{2}{*}{$\X^1 \subseteq \RR^2 \left. \rule{0in}{.20in} \right \{ $} & \multirow{5}{*}{ $\left [\begin{matrix}  4  \\ 2 \\ 12 \\ 5 \\ 3   \end{matrix} \right ] $} &
\multirow{5}{*}{ $\left [\begin{matrix}  8  \\ 4 \\ 4 \\ 1 \\ 1   \end{matrix} \right ] $}  & \multirow{5}{*}{ $\left [\begin{matrix}  12  \\ 6 \\ 4 \\ 2 \\ 8    \end{matrix} \right ] $} & \multirow{5}{*}{ $\left [\begin{matrix}  3  \\ 1 \\ 14 \\ 6 \\ 7  \end{matrix} \right ] $}  \\
& & & &   \\
\multirow{3}{*}{$\X^2 \subseteq \RR^3 \left.   \rule{0in}{.27in} \right \{ $} & & & &   \\
& & & &   \\
& & & &   \\
\end{array}
 \] 

\noindent Then $T_1=\Span \left(  \left [\begin{matrix}  2  \\1 \end{matrix} \right ]    \right)$ is an RSR solution for $\X^1$, and $T_2=\Span \left( \left [\begin{matrix}  12 \\ 5 \\ 3   \end{matrix} \right ], \left [\begin{matrix}  2 \\ 1 \\ 4    \end{matrix} \right ] \right)$ is a RSR solution for $\X^2$. On the other hand, it is easy to check that the partitioned data set $\X=(\X^1,\X^2)$ does not have any SRSR solutions, i.e. $\V$ is $\sigma_0$-semi-stable as a representation of $\K_{4,2}$.
\end{example}

\begin{example}\label{srsr-ex3} This is an example of a partitioned data set $\X=(\X^1, \X^2, \X^3)$ that has an SRSR solution while the data set  $\X^3$ does not have an RSR solution. Consider the partitioned data set given by
\[ \begin{array}{cccccc} 
 			&	 X_1	 	 & X_2 		&       X_3 &	 X_4   &      X_5  \\
			& \verteq	 & \verteq	& \verteq  & \verteq  & \verteq  \\ 
\multirow{4}{*}{$\X^1 \subseteq \RR^3 \left. \rule{0in}{.20in} \right \{ $} & \multirow{7}{*}{ $\left [\begin{matrix}  20  \\ 7 \\ 3 \\ 40 \\ 40 \\ 40 \\ 20 \\ 10   \end{matrix} \right ] $} &
\multirow{7}{*}{ $\left [\begin{matrix}  40  \\ 60 \\ 3 \\ 26 \\ 26 \\ 26 \\ 40 \\ 22   \end{matrix} \right ] $}  & \multirow{7}{*}{ $\left [\begin{matrix}  100  \\ 15 \\ 6 \\ 50 \\ 105 \\ 150 \\ 250 \\120   \end{matrix} \right ] $} & \multirow{7}{*}{ $\left [\begin{matrix}  100  \\ 35 \\ 15 \\ 16 \\ 16 \\ 16 \\  10 \\ 5   \end{matrix} \right ] $} & \multirow{7}{*}{ $\left [\begin{matrix}  100  \\ 100 \\ 15 \\ 100 \\ 200 \\ 27 \\ 300 \\ 140  \end{matrix} \right ] $}  \\
& & & & &  \\
& & & & & \\
\multirow{3}{*}{$\X^2 \subseteq \RR^3 \left.   \rule{0in}{.27in} \right \{ $} & & & & &   \\
& & & & &  \\
& & & & &  \\
\multirow{2}{*}{$\X^3 \subseteq \RR^2 \left.   \rule{0in}{.20in} \right \{ $} & & & & &  \\
 & & & & &  \\
\end{array}
 \] 
 
\noindent Let $(T_1, T_2, T_3)$ be the tuple of subspaces defined by 
\[ T_1 = \left \langle \begin{bmatrix} 20 \\ 7 \\ 3 \end{bmatrix} \right \rangle \subseteq \RR^3, T_2 = \left \langle\begin{bmatrix} 1 \\ 1 \\ 1 \end{bmatrix} \right \rangle \subseteq \RR^3, T_3 = \left \langle \begin{bmatrix} 2 \\ 1 \end{bmatrix} \right \rangle \subseteq \RR^2. \]
Since $\I_T = \{1, 4\}$, we have
\[ |\I_T| > \frac{1 + 1 +1}{3 + 3 + 2} \cdot 5,\] 
thus $(T_1, T_2, T_3)$ is an SRSR solution for $\X$. On the other hand, it is immediate to check that the data set $\X^3$ does not have an RSR solution.  
\end{example}

\begin{example} \label{srsr-ex4}
In this example, we show that there exist partitioned data sets that have several SRSR solutions. We also illustrate the optimal SRSR solution (in the sense of $(\ref{optimal})$) recovered by Algorithms \ref{algGlab} and \ref{algPlab}. Consider the partitioned data set of type $(6, 2, (2,3))$ given by

\[ \begin{array}{cccccc}
 X_1	 	 & X_2 		& X_3      &	 X_4   &   X_5     & X_6 \\ 
 \verteq	 & \verteq	& \verteq  & \verteq  & \verteq  	& \verteq \\
\left [ \begin{matrix} 1 \\ 2 \\ 4 \\ 5 \\ 7 \end{matrix} \right ] & 
\left [ \begin{matrix} 2 \\ 4 \\ 1 \\ 1 \\ 1 \end{matrix} \right ] &
\left [ \begin{matrix} 1 \\ 2 \\ 30 \\ 10 \\ 30 \end{matrix} \right ] &
\left [ \begin{matrix} 1 \\ 2 \\ 4 \\ 4 \\ 4 \end{matrix} \right ] &
\left [ \begin{matrix} 30 \\ 60 \\ 3 \\ 1 \\ 3 \end{matrix} \right ] &
\left [ \begin{matrix} 1 \\ 1 \\ 7 \\ 2 \\ 5 \end{matrix} \right ] \\
 \end{array} 
\] 

This data set has exactly two SRSR solutions. First, consider the tuple of subspaces $(S_1, S_2)$ where
\[S_1 = \left \langle \left [ \begin{matrix} 1 \\ 2 \end{matrix} \right ] \right \rangle \mbox{ and } S_2 = \RR^3.\] Then \[|\I_{(S_1, S_2)}| = |\{1, 2, 3, 4, 5\}| > \frac{1  + 3}{2 + 3} \cdot 6 = \frac{24}{5},\] so $(S_1, S_2)$ is an SRSR solution.

Next, consider the tuple of subspaces $(U_1, U_2)$ where 
\[U_1 = \left \langle \left [ \begin{matrix} 1 \\ 2 \end{matrix} \right ] \right \rangle \mbox{ and } U_2=\left \langle \left [ \begin{matrix} 1 \\ 1 \\ 1 \end{matrix} \right ], \left [ \begin{matrix} 3 \\ 1 \\ 3 \end{matrix} \right ] \right \rangle .\] Then \[ |\I_{(U_1, U_2)}| = |\{2, 3, 4, 5\}| > \frac{1 + 2}{2+3} \cdot 6 = \frac{18}{5},\] so $(U_1, U_2)$ is an SRSR solution.

According to (\ref{optimal}), $(U_1, U_2)$ is the optimal SRSR solution returned by Algorithms \ref{algGlab} and \ref{algPlab} since
\[ |\I_{(U_1, U_2)}| - \frac{1+2}{2+3}\cdot 6 = \frac{2}{5} > |\I_{(S_1,S_2)}| - \frac{1+3}{2+3}\cdot 6 = \frac{1}{5}. \]

\end{example} 

\subsection{Semi-stability via capacity of quiver representations} \label{ssviacap} To address the algorithmic aspects of the semi-stability of quiver representations, we recall the notion of a completely positive operator. 

\begin{definition}
\begin{enumerate}[(i)]
\item Let $\mathcal{A}=\{A_1, \ldots, A_{\ell}\}$ be a collection of real $N \times N$ matrices.  The \emph{completely positive operator (cpo)} associated to $\mathcal{A}$ is the linear operator $T_{\mathcal{A}} \colon \mathbb{R}^{N \times N} \rightarrow \mathbb{R}^{N \times N}$ defined by 
\[X \mapsto  \sum_{i=1}^n A_i X A_i^T \]  The matrices $A_1, \ldots, A_{\ell}$ are called the \emph{Kraus operators} of $T_{\mathcal{A}}$. When no confusion arises, we simply write $T$ instead of $T_{\A}$. 
\item  The \emph{capacity} of a cpo $T$ is defined as 
\[ \cp T = \inf \{ \det \left ( T(Y) \right ) | Y \mbox{ is positive definite with } \det Y = 1 \}  \] 
\end{enumerate}
\end{definition} 

Next we show how to construct a cpo from our representation $\V$ of the complete bipartite quiver $\K_{n,m}$. Let $M=mn$, $N = nD$,  and set
 \[ \s_{\sigma_0}:=\left\{ (i, j, q, r) \middle  \vert 
\begin{array}{l}
i \in [n], j \in [m],\\
\\

(j-1) n < q \leq  j n,\\
\\
(i-1) D < r \leq  i D
\end{array}
\right \}.\] 

For each index $(i,j,q,r) \in \s_{\sigma_0}$, let $V^{i,j}_{q,r}$ be the $M \times N$ block matrix (of size $N \times N$) whose $(q,r)$-block-entry is $v_i^j \in \mathbb{R}^{d_j\times 1}$, and all other block entries are zero matrices of appropriate size. 

\begin{definition} The \emph{cpo associated to the representation} $\V$ is the cpo $T_{\X}$ associated to the collection of matrices 
\[
\A_{\X}:=\{V^{i,j}_{q,r} \mid (i,j,q,r) \in \s_{\sigma_0}\}.
\]
We also refer to $T_{\X}$ as the \emph{cpo associated to the partitioned data set} $\X=(\X^1, \ldots, \X^m)$.
\end{definition}

\begin{remark} We point out $\A_{\X}$ is a linearly independent set of matrices in $\RR^{N \times N}$ with $|\A_{\X}|=n^2mD$.
\end{remark}

\begin{theorem} (see \cite[Theorem 1]{ChiDer-2019}) \label{CD1} Let $\V$ be the representation of the complete bipartite quiver $\mathcal{K}_{n,m}$ defined by the partitioned data set $\mathcal{X}=(\X^1, \ldots, \X^m)$. Then  $\cp(T_{\X}) = 0$ if and only if $\V$ is not $\sigma_0$-semi-stable. 
\end{theorem} 

\begin{remark} The construction above, along with Theorem \ref{CD1} can be formulated for representations of arbitrary bipartite quivers. In the general setting, the cpo above is known as the \emph{Brascamp-Leib operator}, and has deep connections to the celebrated Brascamp-Lieb inequality from Harmonic Analysis.
\end{remark} 

\section{Algorithm \ref{algGlab}} \label{alg-sec} According to Theorem \ref{CD1}, a partitioned data set $\X=(\X^1, \ldots, \X^m)$ has a lower-dimensional subspace structure if and only if the capacity of $T_{\X}$ is zero. To check whether or not the capacity is zero, we will make use of an algorithm originally published in \cite{Gurv04}, known as Algorithm \ref{algGlab}. In order to state Algorithm \ref{algGlab}, we first recall some definitions. 

\begin{definition} Let $T_{\A}$ be a cpo with Kraus operators $\A = \{A_1, \ldots, A_{\ell}\} \subseteq \RR^{N \times N}$.
\begin{enumerate}[(i)]

\item The \emph{dual} of $T_{\A}$, denoted $T_{\mathcal{A}}^*$, is the linear operator defined by \[X \mapsto  \sum_{i=1}^{\ell} A_i^T X A_i  \]

\item We call $T_{\mathcal{A}}$ (and $T_{\mathcal{A}}^*$) \emph{doubly stochastic} if $$ \dsp T_{\A}(I) = T_{\mathcal{A}}^*(I) = I.$$ 
 
\item The \emph{distance of $T_{\A}$ to double stochasticity} is defined by:

\[\text{ds}(T_{\A}) : = \tr\left((T_{\A}(I) - I)^2 \right ) + \tr\left( (T_{\A}^*(I) - I)^2 \right )\]
\end{enumerate}
\end{definition} 

The next definition gives two specific instances of a more general process known as \emph{operator scaling}. It leads to an efficient algorithm for testing for the positivity of the capacity of a cpo. 

\begin{definition} Let $T$ be a cpo with Kraus operators $\{A_1, \ldots, A_{\ell} \} \subseteq \RR^{N \times N} $ such that $T(I)$ and $T^{*}(I)$ are both invertible. 
\begin{enumerate}[(i)]
	\item The \emph{right normalization of $T$}, $T_R$,  is the cpo with Kraus operators $$\dsp A_1\cdot T^*(I)^{-1/2}, \ldots,  A_{\ell} \cdot T^*(I)^{-1/2}.$$ 
	\item The \emph{left normalization of $T$}, $\dsp T_L$,  is the cpo with Kraus operators  $$\dsp T(I)^{-1/2}A_1, \ldots, T(I)^{-1/2}A_{\ell}.$$ 
\end{enumerate}

We point out that $T^*_R(I)=\dsp T_L(I) = I $.	
\end{definition}

\noindent We are now ready to state Algorithm \ref{algGlab}, adapted slightly to our context from \cite[page 47]{GarGurOliWig-2015}. We point out that their version is concerned with estimating capacity within a given multiplicative error, whereas we are simply concerned with checking whether or not the capacity is positive. It has been proved that Algorithm \ref{algGlab} runs in (deterministic) polynomial time.

\vspace{.1in} 

\begin{alg}[label={algGlab}]{}{}

\textbf{Input}: A cpo $T$ whose Kraus operators (of size $N \times N$)  lie over $\mathbb{Z}$ and  $M$ the maximum magnitude of each entry.\\
\textbf{Output}: The positivity of $\capa(T)$. \\

\begin{enumerate} 
\item If either $T(I)$ or $T^*(I)$ is singular then output $\cp(T) = 0$. Otherwise, continue.
\item For each integer $j\geq 1$, perform (alternatively) right and left normalization on $T=T_0$ with $T_j$ the operator after $j$-steps. 
\item If $ \dsp \text{ds}(T_j) \leq \frac{1}{4N^3}$ after $j$ steps,  output $\cp(T) > 0$. 
\item If $ \dsp \text{ds}(T_j) > \frac{1}{4N^3}$ after $\dsp 4N^3 \left ( 1 + 10N^2 \log(MN) \right )$ steps, output $\cp(T) = 0$.

\end{enumerate} 

\end{alg} 

\vspace{.1in}

Even though the algorithm is stated over $\ZZ$, it is valid over $\QQ$, since multiplying the Kraus operators by a (non-zero) scalar does not affect the positivity of the capacity. While we have worked over $\mathbb{R}$, almost all data will be rational points, so there is no loss in working over $\mathbb{Q}$. For proof of correctness, we refer the reader to \cite{GarGurOliWig-2015}. See also \cite{Gurv04}. 

\begin{remark} Algorithm G has been developed by Gurvits in \cite{Gurv04} to solve Edmond's Problem for certain classes of matrices.  Edmond's Problem, originally posed in \cite{Edm-1967}, asks whether or not the span of a collection of (square) matrices contains a non-singular matrix.  In \cite{ChiKli-Edmonds-2020}, we show how Edmond's problem can be expressed via orbit semigroups of quiver representations.
\end{remark} 

\section{Shrunk subspaces and rank decreasing witnesses} \label{ssrkdc}

Given a partitioned data set $\mathcal{X}$, using Theorem \ref{CD1} and Algorithm \ref{algGlab}, we can check in deterministic polynomial time whether $\X$ has an SRSR solution (equivalently, whether $\cp(T_\X)=0$). When this is the case, we still need to find a way to recover this lower-dimensional subspace structure. In order to do this, we consider shrunk subspaces and second Wong sequences which we explain below.

\begin{definition} Let $\ar = \{A_1, \ldots A_{\ell} \}$ an $\ell$-tuple of $N \times N$ real matrices and $T_{\ar}$ the cpo with Kraus operators $A_1, \ldots, A_{\ell } $. Let $c$ be a non-negative integer. A subspace $U \subseteq \mathbb{R}^N$ is called a \emph{$c$-shrunk subspace} for $\A$ (or $T_{\A}$) if 
\[ \dim U-\dim \left( \sum_{i=1}^{\ell} A_i(U) \right) \geq c.\]
We say that $U$ is a \emph{shrunk subspace} for $\A$ if $U$ is a $c$-shrunk subspace for some $c>0$.
\end{definition}

\noindent We point out that shrunk subspaces have also been studied by D. Eisenbud and J. Harris in \cite{EisHar}, though they refer to them as \emph{compression spaces}.

\begin{definition} A cpo $T$ on $\mathbb{R}^{N\times N}$ is said to be \emph{rank-decreasing} if there exists an $N \times N$ positive semi-definite matrix $Y$ such that $\rk T(Y)<\rk Y$. In this case, we call $Y$ a \emph{rank-decreasing witness} for $T$. 
\end{definition}

\begin{prop} \label{equiv-cap-0} Let $\X$ be a partitioned data set, $\A_{\X} = \{A_1, \ldots, A_{\ell}\}$ the Kraus operators associated to $\X$, and $T_{\X}$ the corresponding cpo. Then the following statements are equivalent:
\begin{enumerate}[(i)]
\item $\V$ is not $\sigma_0$-semi-stable;
\item $\cp(T_{\X}) = 0$;
\item $T_{\X}$ is rank-decreasing;
\item $T_{\X}$ has a shrunk subspace.
\end{enumerate}
\end{prop}

\begin{proof}  By Theorem \ref{CD1}, we know that  $(i)$ and $(ii)$ are equivalent. The equivalence of $(ii)$ and $(iii)$ follows from \cite[Lemma 4.5]{Gurv04}.

For $(iii)$ implies $(iv)$, let $Y$ be a rank-decreasing witness for $T_{\X}$ of rank $r$. Then we can write it as
$$
Y=\sum_{j=1}^r \lambda_j u_j\cdot u_j^T
$$
where $\lambda_1, \ldots, \lambda_r>0$ and the vectors $u_1, \ldots, u_r$ form an orthonormal set. Setting $U:=\Span(u_1, \ldots, u_r)$, we get that $\dim U=\rk(Y)$, and 
\begin{align*} 
\dim \left( \sum_{i=1}^{\ell} A_i(U) \right) & =\dim \Span \left( A_iu_j\mid i\in[\ell], j \in [r] \right) \\
& = \rk \left ( \sum_{i,j} A_iu_j (A_iu_j)^T\right )=\rk T_{\X}(Y).
\end{align*} 
Thus \[\dim U-\dim \left( \sum_{i=1}^{\ell} A_i(U) \right) =\rk(Y)-\rk(T_{\X}(Y))>0,\]
i.e. $U$ is a shrunk subspace for $T_{\X}$.

For $(iv)$ implies $(iii)$, if $U'$ is a shrunk subspace for $T_{\X}$, let $Y'$ be the matrix of orthogonal projection onto $U'$. Write $Y'=Q\cdot Q^T$ where $Q$ is a matrix whose columns form an orthonormal basis for $U'$. Then $\rk(Y') = \dim U'$, and since $T_{\X}(Y') = \sum_{i = 1}^{\ell} (A_iQ) (A_iQ)^T$, it follows that
\[ \Ima(T_{\X}(Y')) \subseteq \sum_{i = 1}^{\ell} \Ima(A_i Q) = \sum_{i=1}^{\ell} A_i(U') \]
which gives
\[ \rk(T_{\X}(Y')) \leq \dim \left( \sum_{i=1}^{\ell} A_i(U') \right) < \dim U' = \rk(Y').\]  
Thus
$$
\rk(Y')-\rk(T_{\X}(Y'))\geq \dim U'-\dim \left( \sum_{i=1}^{\ell} A_i(U') \right)>0,
$$
i.e. $Y'$ is a rank-decreasing witness for $T_{\X}$. 
\end{proof} 

We will use rank-decreasing witnesses for $T_{\X}$ to construct an SRSR solution for $\X$ (see Lemma \ref{datarec}). The proof above shows that finding a rank-decreasing witness for $T_{\X}$ is equivalent to finding a shrunk-subspace. In order to compute shrunk-subspaces, we recall the notion of \emph{second Wong sequences}.

\begin{definition}  Let $\mathcal{A} = \{A_1, \ldots A_{\ell} \} \subset  \mathbb{R}^{N \times N}  $. 
For $B \in \mathbb{R}^{N \times N}$, the \emph{second Wong sequence} associated to $(\mathcal{A}, B)$ is the sequence of subspaces
\begin{align*}
W_0 & = 0  \\
W_1 & = \mathcal{A}(B^{-1}(W_0)) \\
W_2 & = \mathcal{A}(B^{-1}(W_1)) \\
 & \vdots \\
 W_k & = \mathcal{A}(B^{-1}(W_{k-1})) \\
\end{align*}
where for each $j$, $\mathcal{A}(B^{-1}(W_j)) = \langle A_1(B^{-1}(W_j)), \ldots, A_{\ell}(B^{-1}(W_j)) \rangle $
and $B^{-1}(W_j)$ is the preimage of $W_j$ under $B$. 
\end{definition}

\noindent It is not difficult to see that the $W_i$ lie in a chain, i.e. 
\[ W_0 \subseteq W_1 \subseteq W_2 \subseteq \ldots \]
Since we are working over finite-dimensional vector spaces, the second Wong sequence must converge, i.e. there exists a positive integer $L \leq N$ such that
\[W_0 \subseteq W_1 \subseteq \ldots \subseteq W_{L-1} \subseteq W_L = W_{L+1} = \ldots\]
 $W_L$ is called the \emph{limit} of the associated second Wong sequence and is denoted by $W^*$. 

The next result shows how shrunk subspaces can be computed from second Wong sequences. 

\begin{theorem}{\cite[Theorem 1 \& Lemma 9]{IKQS15}}\label{IKQS-thm-shrunk}
Let $\mathcal{A} = \{A_1, \ldots A_{\ell} \} \subset \RR^{N \times N}$ and $B \in \langle \mathcal{A} \rangle$. Then the following statements are equivalent:
\begin{enumerate}
\item $W^* \subseteq \Ima \; B$;

\item $B^{-1}(W^*)$ is a $\crk(B)$-shrunk subspace.
\end{enumerate}

\noindent If the operators $A_1, \ldots, A_{\ell}$ are assumed to be of rank one then $(1)$ and $(2)$ are further equivalent to:

\begin{enumerate}
\item[(3)] $B$ has maximal rank among the matrices in $\langle \A \rangle$. 
\end{enumerate}
In fact, under the assumption that $A_1, \ldots, A_{\ell}$ are of rank one, there exists a deterministic polynomial time algorithm to find a maximal rank matrix $B$ in $\langle \A \rangle$.
\end{theorem}

\begin{rmk} We point out that in general $(1)$ or $(2)$ in Theorem \ref{IKQS-thm-shrunk} implies that the matrix $B$ is of maximal rank. 
\end{rmk}

In the context of SRSR, since the Kraus operators for $T_{\X}$ are all of rank one, Theorem \ref{IKQS-thm-shrunk} tells us that computing a shrunk subspace comes down to constructing a matrix $B \in \langle \A_{\X} \rangle$ of maximal rank. Furthermore it also provides a deterministic polynomial time algorithm for finding such a maximal rank matrix. However, from a practical standpoint, it is rather difficult to implement. In what follows, we are going to describe an algorithm that constructs a shrunk subspace with high probability.

\begin{prop} \label{prop-prob-alg} Let $\A_{\X}=\{A_1, \ldots, A_{\ell}\}$ be the set of Kraus operators for $T_{\X}$ where $\ell=n^2mD$ and let $S \subseteq \RR$ be any finite set of size at least $  \frac{2n^2D^2}{\epsilon}$ with $\epsilon>0$. 

If $\alpha_1, \ldots, \alpha_{\ell}$ are chosen independently and uniformly at random from $S$ and $B:=\alpha_1 A_1+\ldots+\alpha_{\ell} A_{\ell} \in \langle \A_{\X} \rangle$ then
\[
\mathbb{P}\left( B^{-1}(W^*) \text{~is a~} \crk(B)\text{-shrunk-subspace for~} \A_{\X}\right) \geq 1-\epsilon.
\]
\end{prop}

\begin{proof} Let $\vv$ be the subspace of $\RR^{N \times N}$ spanned by $\A_{\X}$. Since $\A_{\X}$ is a linearly independent set of matrices, it is a basis for $\vv$, and thus we can identify $\vv$ with $\RR^{\ell}$. 

Let us denote by $r$ the maximal rank of the matrices in $\vv$, and let $P \in \RR[t_1, \ldots, t_{\ell}]$ be the non-zero polynomial defined by
\[
P(t_1, \ldots, t_{\ell})=\sum \det(A')^2
\]
where the sum is over all $r \times r$ submatrices $A'$ of $t_1A_1+\ldots+t_{\ell}A_{\ell}$. Note that the total degree of this polynomial is at most $2r^2\leq 2N^2=2n^2D^2$. 

Then, for any $(\alpha_1, \ldots, \alpha_{\ell}) \in \RR^{\ell}$, we have that
$$
\alpha_1A_1+\ldots+\alpha_{\ell}A_{\ell} \in \vv \text{~is not of maximal rank~} \Longleftrightarrow P(\alpha_1, \ldots, \alpha_{\ell})=0.
$$
At this point we can use the Schwartz-Zippel Lemma \cite{Schwartz-lemma-1980, Zippel-lemma-1979} to conclude that
\begin{equation} \label{S-Z-ineq}
\mathbb{P}(B \text{~is not of maximal rank}) \leq \frac{\deg(P)}{|S|}\leq \epsilon
\end{equation}

It now follows from Theorem \ref{IKQS-thm-shrunk} and $(\ref{S-Z-ineq})$ that
$$
\mathbb{P}\left( B^{-1}(W^*) \text{~is a~} \crk(B)\text{-shrunk-subspace for~} \A_{\X}\right)=1-\mathbb{P}\left(~ \parbox{5em}{$B$ is not of\\ maximal rank} \right) \geq 1-\epsilon,
$$
and this finishes the proof.
\end{proof}

We are thus led to the following simple, efficient probabilistic algorithm for computing shrunk-subspaces for $\A_{\X}$. This algorithm returns an output with high probability 
by Proposition \ref{prop-prob-alg}. In what follows, we denote by $\crk(\X)$ the corank of a matrix $B \in \langle \A_{\X} \rangle$ of maximal rank.

\vspace{.1in}

\begin{algP}[label={algPlab}]{}{}
\textbf{Input}: Kraus operators $\A_{\X}=\{A_1, \ldots, A_{\ell}\}$ with ${\ell}=n^2mD$, and a sample set $S \subseteq \RR$ of size at least $ \displaystyle \frac{2n^2D^2}{\epsilon}$. \\
\textbf{Output}: $\crk(\X)$-shrunk-subspace for $\A_{\X}$ 
\begin{enumerate}
\vspace{.1in}
\item Choose $\alpha_1, \ldots, \alpha_{\ell}$ independently and randomly distributed from $S$, and set $B:=\alpha_1 A_1+\ldots+\alpha_{\ell} A_{\ell}$.
\vspace{.1in}
\item Compute $W^*$, the limit of the second Wong sequence associated to $(\A_{\X}, B)$.
\vspace{.1in}
\item If $W^* \subseteq \Ima B$, output $B^{-1}(W^*)$ as a $\crk(\X)$-shrunk subspace for $\A_{\X}$. Otherwise, return to (1).
\end{enumerate}
\end{algP}

\vspace{.1in}

\begin{remark} $(1)$~We point out that Proposition \ref{prop-prob-alg} can be restated to say that for a randomly chosen matrix $B$ from $\vv:=\langle \A_{\X} \rangle$, the subspace $B^{-1}(W^*)$ is a $\crk(B)$-shrunk-subspace for $\A_{\X}$ with probability \emph{one}. Indeed, as in the proof of Proposition \ref{prop-prob-alg}, we have the non-zero polynomial $P \in \RR[t_1, \ldots, t_{\ell}]$ such that for any $(\alpha_1, \ldots, \alpha_{\ell}) \in \RR^{\ell}$,
\[
\alpha_1A_1+\ldots+\alpha_{\ell}A_{\ell} \in \vv \text{~is not of maximal rank~} \Longleftrightarrow P(\alpha_1, \ldots, \alpha_{\ell})=0.
\]
So the subset of $\vv$ consisting of all those matrices which are not of maximal rank is identified with the zero set of $P$ in $\RR^{\ell}$. This zero set, denoted by $\mathbb{V}(P)$, has Lebesgue measure zero in $\RR^{\ell}$. 

Now, let $(\alpha_1, \ldots, \alpha_{\ell}) \in \RR^{\ell}$ be a coefficient vector drawn from any probability distribution absolutely continuous with respect to the Lebegue measure on $\RR^{\ell}$ (e.g. Gaussian distribution). If $B:=\alpha_1A_1+\ldots+\alpha_{\ell}A_{\ell} \in \vv$ then
\[
\mathbb{P}(B \text{~is not of maximal rank})=\mathbb{P}((\alpha_1, \ldots, \alpha_{\ell}) \in \mathbb{V}(P))=0,
\]
proving that a randomly chosen matrix $B$ from $\langle \A_{\X} \rangle$ has maximal rank with probability one. This fact combined with Theorem \ref{IKQS-thm-shrunk} yields the desired claim.

\bigskip 

$(2)$~Let $\vv$ be an arbitrary matrix space of dimension $\ell \geq 1$. Then, using the same arguments as above, one can prove the following statements.

\begin{enumerate}[(a)]
\item There exists a random process (based on the Schwartz-Zippel Lemma) that selects a random matrix from $\vv$ of maximal rank with high probability (see the proof of Proposition \ref{prop-prob-alg}).

\item There exists a random process (based on absolutely continuous probability distributions) that selects a random matrix from $\vv$ of maximal rank with probability one.
\end{enumerate}

From an implementation standpoint, the random process in $(a)$ is more practical than that in $(b)$. 

\end{remark}

\section{Finding an SRSR solution} 
In this section, $\X=(\X^1, \ldots, \X^m)$ is a partitioned data set with $\X^j=\{v^j_1, \ldots, v^j_n\} \subseteq \RR^{d_j}$, $j \in [m]$. Let $\V$ be the representation of the complete bipartite quiver $\K_{n,m}$ corresponding to $\X$. We denote by $T_{\X}$ the cpo associated to the representation $\V$. We will be working with the weight $\sigma_0$ of $\K_{n,m}$ defined by
$$
\sigma_0(x_i)=D, \forall i \in [n], \text{~and~} \sigma_0(y_j)=-n, \forall j \in [m],
$$
where $D=\sum_{j \in [m]} d_j$. 

We begin by explaining how to construct a subrepresentation of $\V$ from a given rank-decreasing witness for $T_{\X}$.

\begin{lemma}\label{datarec}  Let $\V$ be the representation of $\K_{n,m}$ associated to the partitioned data set $\X$ such that $\cp(T_{\X})=0$. Let $Y=(y_{q,r})_{q,r}$ be an $N \times N$ rank-decreasing witness for $T_{\X}$ where $N=nD$. Set
\[ \I = \{ i \in [n] \mid y_{r,r}\neq 0 \text{~for some~} r \in \I_i \}, \]
where $\I_i=[(i-1)D+1, iD]$ for every $i \in [n]$. Let  $W$ be the subrepresentation of $\V$ defined by 
\[
 W(x_i) = \begin{cases} \mathbf{0} & i \notin \I \\ \RR & i \in \I \end{cases} \; \; \mbox{ for } i \in [n], \;\;  and \;\; W(y_j) = \Span \left  ( v_i^j | i \in \I  \right )   \; \; \mbox{ for } j \in [m]. 
\]
Then $$\sigma_0 \cdot \ddim W\geq \rk(Y)-\rk(T_{\X}(Y))>0.$$
\end{lemma}  

\begin{proof} It is clear that
\begin{equation} \label{rkeq1}  
\rk(Y) \leq \sum_{i \in [n]} \sum_{r \in \I_i} \rk(y_{r,r}) \leq \sum_{i \in [n]} D \cdot \dim W(x_i).
\end{equation} 

\noindent On the other hand, $T_{\X}(Y)$ turns out to be an $(mn) \times (mn)$ block diagonal matrix whose $(q,q)$-block entry is 
$$
\sum_{i=1}^n v_i^j\cdot \left(\sum_{r \in \I_i} y_{r,r}\right) \cdot (v_i^j)^T 
$$
for every $q \in [(j-1)n+1, jn] $ and $j \in [m]$. Thus we get that
\begin{align*} \rk(T_{\X}(Y) ) & = \sum_{j=1}^m n \rk \left ( \sum_{i=1}^n \left ( v_i^j \left(\sum_{r \in \I_i} y_{r,r}\right) (v_i^j)^T \right ) \right ) \numberthis \label{rkeq2} \\
& = \sum_{j=1}^m n \rk \left ( \sum_{i,r} \sqrt{y_{r,r}} v_i^j (\sqrt{y_{r,r}}v_i^j)^T\right ) \\
& = \sum_{j=1}^m n \dim \Span \left ( \sqrt{y_{r,r}} v_i^j \mid i \in [n], r \in \I_i \right ) \\
& = - \sum_{j=1}^m  \sigma_0(y_j) \dim W(y_j)
 \end{align*} 

From (\ref{rkeq1}) and (\ref{rkeq2}), it follows that 

\begin{equation} 
\sum_{i \in [n]} \sigma_0(x_i) \dim W(x_i) \geq \rk(Y) > \rk(T_{\X}(Y)) = - \sum_{j \in [m]}  \sigma_0(y_j) \dim W(y_j),
\end{equation}
which proves our claim. 
\end{proof}

We are now ready to prove our main result.

\begin{proof}[Proof of Theorem \ref{main-thm}] Let $\X$ be a partitioned data set, $\V$ the associated representation of $\K_{n,m}$, and $T_{\X}$ the associated cpo. Since $\A_{\X}$ consists of matrices of rank one, $\A_{\X}$ has the so-called Edmonds-Rado property (see \cite{Lov-ERP-R1-1989}), meaning that $\cp(T_{\X})>0$ if and only if $\crk(\X)=0$. (Recall that $\crk(\X)$ denotes the corank of a matrix $B \in \langle \A_{\X} \rangle$ of maximal rank.) It follows from Theorem \ref{CD1} and Proposition \ref{equiv-cap-0} that the following statements are equivalent:
\begin{itemize}
\item $\X$ has an SRSR solution;
\item $\cp(T_{\X})=0$;
\item $\crk(\X)>0$.
\end{itemize}

Thus, to determine whether or not $\X$ has an SRSR solution, we can use either Algorithm \ref{algGlab}, which runs in deterministic polynomial time, or the probabilistic Algorithm \ref{algPlab}.

Assume now that $\X$ has an SRSR solution. Then we can compute a $\crk(\X)$-shrunk subspace $U$ for $T_{\X}$ by the deterministic polynomial time algorithm in Theorem \ref{IKQS-thm-shrunk} or by the probabilistic Algorithm \ref{algPlab}. More precisely, the output $U$ is of the form $B^{-1}(W^*)$ where $B \in \langle \A_{\X} \rangle$ is of maximal rank and $W^*$ is the limit of the second Wong sequence associated to $(\A_{\X}, B)$.

Next, let $Y=Q\cdot Q^T$ where $Q$ is a matrix whose columns form an orthonormal basis for $U$. Then, according to Proposition \ref{equiv-cap-0}, $Y$ is a rank-decreasing witness for $T_{\X}$. By Lemma \ref{datarec}, we can finally use $Y$ to find a subspace $W \subseteq \V$ such that 
\begin{equation} 
\label{subrep-max-crk-eqn} \sigma_0 \cdot \ddim W \geq \rk(Y)-\rk(T_{\X}(T)) \geq \crk(\X)> 0. 
\end{equation}
For each $j \in [m]$,  define $T_j = W(y_j)$. Then (\ref{subrep-max-crk-eqn}) yields
\[|\I_T| > \frac{ \sum_{j \in [m]} \dim T_j}{D} n, \] 
and so $(T_1, \ldots, T_m)$ is an SRSR solution. 

It remains to show that $(T_1, \ldots, T_m)$ is an optimal SRSR solution So let $(T'_1, \ldots, T'_m)$ be any SRSR solution for $\X$, and let $W'$ be the corresponding subrepresentation of $\V$ (see Remark \ref{SRSR-gen-ex}). Then there exists a $c'$-shrunk subspace for $T_{\X}$ with $c'=\sigma_0\cdot \ddim W'$. Indeed, let $Y'$ be the $N\times N$ diagonal matrix whose $r$th diagonal entry is one if $r \in [(i-1)D+1, iD]$ with $i \in \I_{T'}$, and zero otherwise. It is easy to check that $Y'$ is a rank-decreasing witness for $T_{\X}$ with
$$
\rk(Y')-\rk(T_{\X}(Y')) \geq \sigma_0 \cdot \ddim (W').
$$
From Proposition \ref{equiv-cap-0} it follows that if $U'$ denotes the shrunk subspace corresponding to $Y'$ then 
\begin{equation} \label{subrep-shrunk-eqn}
\dim U'-\dim \left(\sum_{i=1}^l A_i(U') \right)= \rk(Y')-\rk(T_{\X}(Y'))\geq \sigma_0 \cdot \ddim (W'),
\end{equation}
where $\A_{\X}=\{A_1, \ldots, A_l\}$, i.e. $U'$ is a $\sigma_0\cdot \ddim (W')$-shrunk subspace.

On the other hand, we know that
\begin{equation}\label{crk-disc-ineq}
\crk(\A_{\X}) \geq \max \{c \in \NN \mid \text{there exists a~}c\text{-shrunk subspace for~} \A_{\X}\}
\end{equation}
(This inequality holds for any collection of matrices $\A$ with the maximum on the right hand side also known as the discrepancy of $\A$; see for example \cite{IKQS15}.) Using $(\ref{subrep-max-crk-eqn})$, $(\ref{subrep-shrunk-eqn})$, and $(\ref{crk-disc-ineq})$ we get that
\begin{align*}
\sigma_0\cdot \ddim W & \geq \sigma_0 \cdot \ddim W'\\  
&\rotatebox[origin=c]{-90}{$\Leftrightarrow$} \notag \\ 
D \cdot  |\I_T|- n \cdot \left( \sum_{j=1}^m \dim T_j \right)& \geq D \cdot  |\I_{T'}|- n \cdot \left( \sum_{j=1}^m \dim T'_j \right) \\  
&\rotatebox[origin=c]{-90}{$\Leftrightarrow$} \notag \\ 
|\I_T|- {\sum_{j=1}^m \dim T_j \over D} \cdot n& \geq  |\I_{T'}|- {\sum_{j=1}^m \dim T'_j \over D} \cdot n.
\end{align*}
This finishes the proof.
\end{proof} 

\section{Semi-stability of quiver representations}\label{general-semi-stab-sec}
In this section, we address the general problem of effectively deciding whether a quiver representation is semi-stable (with respect to a weight) and, if it is not, finding a subrepresentation that certifies that the representation is not semi-stable.

Let $Q=(Q_0, Q_1, t, h)$ be an arbitrary connected bipartite quiver (not necessarily complete). This means that $Q_0$ is the disjoint union of two subsets $Q^+_0=\{x_1, \ldots, x_n\}$ and $Q^-_0=\{y_1, \ldots, y_m\}$, and all arrows in $Q$ go from $Q^+_0$ to $Q^-_0$. (We do allow multiple arrows between the vertices of $Q$.)  

Let $\dd \in \NN^{Q_0}$ be a dimension vector of $Q$. The representation space of $\dd$-dimensional representations of $Q$ is the vector space
$$
\rep(Q,\dd):=\prod_{a \in Q_1} \RR^{\dd(ha)\times \dd(ta)}.
$$
It is clear that any $V \in \rep(Q,\dd)$ gives rise to a representation of $Q$ of dimension vector $\dd$ and vice versa. 

Let $\sigma \in \ZZ^{Q_0}$ be a weight of $Q$ such that $\sigma$ is positive on $Q^{+}_0$, negative of $Q^{-}_0$, and
\begin{equation} \label{ortho-wts-eqn}
\sigma \cdot \dd=\sum_{z \in Q_0}\sigma(z)\dd(z)=0.
\end{equation}
Define
$$\sigma_{+}(x_i)=\sigma(x_i), \forall i \in [n], \text{~and~}
\sigma_{-}(y_j)=-\sigma(y_j), \forall j \in [m].
$$
Then $(\ref{ortho-wts-eqn})$ is equivalent to
$$
N:=\sum_{i=1}^n \sigma_{+}(x_i)\dd(x_i)=\sum_{j=1}^m \sigma_{-}(y_j)\dd(y_j).
$$

Let $M:=\sum_{j=1}^m \sigma_{-}(y_j) \text{~and~} M':=\sum_{i=1}^n \sigma_{+}(x_i).$
For each $j \in [m]$ and $i \in [n]$, define 

$$
\jj:=\{q \in \ZZ \mid \sum_{k=1}^{j-1} \sigma_{-}(y_k) < q \leq  \sum_{k=1}^{j} \sigma_{-}(y_k) \},
$$
and 
$$
\ii:=\{r \in \ZZ \mid \sum_{k=1}^{i-1} \sigma_{+}(x_k) < r \leq  \sum_{k=1}^{i} \sigma_{+}(x_k) \}.
$$
In what follows, we consider $M \times M'$ block matrices of size $N \times N$ such that for any two indices $q \in \jj$ and $r \in \ii$, the $(q,r)$-block-entry is a matrix of size $\dd(y_j)\times \dd(x_i)$. Set
$$
\s_{\sigma}:= \{(i,j,a,q,r) \mid i \in [n], j \in [m],\\ a \in \ar_{i,j}, \\ q \in \jj, r \in \ii \},
$$
where $\ar_{i,j}$ denotes the set of all arrows in $Q$ from $x_i$ to $y_j$, $i \in [n]$ and $j \in [m]$. 

Let $V \in \rep(Q,\dd)$ be a $\dd$-dimensional representation. For each index $(i,j,a,q,r) \in \s_{\sigma}$, let $V^{i,j,a}_{q,r}$ be the $M \times M'$ block matrix of size $N \times N$ whose $(q,r)$-block-entry is $V(a) \in \RR^{\dd(y_j)\times \dd(x_i)}$, and all other entries are zero. We denote by $\A_{V, \sigma}$ the set of all matrices $V^{i,j,a}_{q,r}$ with $(i,j,a,q,r) \in \s_{\sigma}$.

The Brascamp-Lieb operator associated to the quiver datum $(V, \sigma)$, defined in \cite{ChiDer-2019}, is the cpo $T_{V, \sigma}$ with Kraus operators $V^{i,j,a}_{q,r}$, $(i,j,a,q,r) \in \s_{\sigma}$, i.e.
\begin{align*}
T_{V,\sigma}(X):=\sum_{(i,j,a,q,r)} V^{i,j,a}_{q,r} \cdot X \cdot (V^{i,j,a}_{q,r})^T, \forall X \in \RR^{N \times N}.
\end{align*}

It has been proved in \cite[Theorem 1]{ChiDer-2019} that
$$
V \text{~is~} \sigma\text{-semi-stable} \Longleftrightarrow \cp(T_{V, \sigma})>0.
$$
Consequently, as indicated in \cite[Section 1.2]{ChiDer-2019}, one can use Algorithm \ref{algGlab} to test whether $V$ is $\sigma$-semi-stable or not. 

\begin{rmk}\label{Algo-P-general-reps-rmk1} In the context of arbitrary quivers and dimension vectors, when it comes to Algorithm \ref{algPlab}, even if $B$ is chosen to have maximal rank, it is not guaranteed that $W^* \subseteq \Ima B$. This is because for arbitrary dimension vectors, $\A_{V, \sigma}$ may contain matrices that are not of rank one. Thus Algorithm \ref{algPlab} might not produce an output.
\end{rmk}

The following lemma allows us to go back and forth between subrepresentations on the one hand, and  rank-decreasing witnesses and $c$-shrunk subspaces on the other. It plays a key role in the proof of Theorem \ref{main-thm-2}

\begin{lemma} \label{subrep-witness-shrunk-general-lemma} Let $Q$ be a bipartite quiver and $(V, \sigma)$ a quiver datum as above. Let $\A_{V, \sigma}=\{A_1, \ldots, A_L\}$ be the set of Kraus operators associated to $(V, \sigma)$.
\begin{enumerate}
\item Let $W$ be a subrepresentation of $V$. For each $i \in [n]$, after choosing an orthonormal basis $u^i_1, \ldots, u^i_{\ff(i)}$  for $W(x_i) \subseteq \RR^{\dd(x_i)}$, define
$$
Y_r:=\sum_{l=1}^{\ff(i)} u^i_l \cdot (u^i_l)^T, \forall r \in \I^{+}_i.
$$ 
Then  $Y:=\bigoplus_{i \in [n]} \bigoplus_{r \in \I^+_i} Y_r \in \RR^{N \times N}$ is a positive semi-definite matrix such that
$$
\rk(Y)-\rk(T_{V, \sigma}(Y)) \geq \sigma \cdot \ddim W.
$$
Moreover, if $U \subseteq \RR^N$ is the subspace associated to $Y$ (see Proposition \ref{equiv-cap-0}) then
$$
\dim U - \dim \left( \sum_{i=1}^L A_i(U) \right) \geq \sigma \cdot \ddim W.
$$

\item Let $Y \in \RR^{N \times N}$ be a positive semi-definite matrix. Viewing $Y$ as an $M' \times M'$ block matrix, denote its block diagonal entries by $Y_r \in \RR^{\dd(x_i)\times \dd(x_i)}$ with $r \in \I^{+}_i$ and $i \in [n]$. For each such $r$ and $i$, write 
$$
Y_r=\sum_{l=1}^{d_{i,r}} \lambda^{i,r}_l u^{i,r}_l \cdot (u^{i,r}_l)^T,
$$
where $d_{i,r}$ is the rank of $Y_r$, the $\lambda^{i,r}_l$ are positive scalars, and the $u^{i,r}_l$ form an orthonormal set of vectors in $\RR^{\dd(x_i)}$. Define
$$
W(x_i):=\Span\left( \sqrt{\lambda^{i,r}_l} u^{i,r}_l \mid r \in \I^{+}_i, l \in [d_{i,r}]\right), \forall i \in [n], 
$$
and 
$$
W(y_j):=\sum_{i \in [n]} \sum_{a \in \ar_{i,j}} V(a)(W(x_i)), \forall j \in [m].
$$
Then $W:=(W(x_i), W(y_j))_{i \in [n], j \in [m]}$ is a subrepresentation of $V$ such that
$$
\sigma \cdot \ddim W \geq \rk(Y)-\rk(T_{V, \sigma}(Y)).
$$
\item Let $U \subseteq \RR^N$ be a $c$-shrunk subspace for $T_{V,\sigma}$ for some $c \in \NN$, and let $Y \in \RR^{N \times N}$ be the matrix of orthogonal projection onto $U$. If $W$ is the subrepresentation of $V$ associated in $(2)$ to $Y$ then
$$
\sigma \cdot \ddim W \geq c.
$$
\end{enumerate}
\end{lemma}

\begin{proof} $(1)$ We have that $T_{V, \sigma}(Y)$ is the $M \times M$ block diagonal matrix whose $(q,q)$-block-diagonal entry is
$$
\sum_{i=1}^n~ \sum_{a \in \ar_{i,j}} V(a)\left(\sum_{r \in \I^{+}_i} Y_r  \right) V(a)^T,
$$
for all $q \in \I^{-}_j$ and $j \in [m]$. Thus we get that
\begin{align*}
\rk(T_{V,\sigma}(Y)) & =\sum_{j=1}^m \sigma_{-}(y_j) \rk \left( \sum_{i=1}^n \sum_{a \in \ar_{i,j}} \sum_{l,r} V(a)u^i_l (V(a)u^i_l )^T   \right) \\
& = \sum_{j=1}^m \sigma_{-}(y_j) \dim \left( \sum_{i=1}^n \sum_{a \in \ar_{i,j}} V(a)(W(x_i))   \right) \leq \sum_{j=1}^m \sigma_{-}(y_j) \dim W(y_j). \\
\end{align*} 
From this it follows that
$$
\rk(Y)-\rk(T_{V, \sigma}(Y)) \geq \sum_{i=1}^n \sigma_{+}(x_i)\dim W(x_i)-\sum_{j=1}^m \sigma_{-}(y_j) \dim W(y_j)=\sigma \cdot \ddim W.
$$

\noindent $(2)$ We have
\begin{equation}\label{eqn1-lemma-quivers}
\rk(Y) \leq \sum_{i=1}^n \sum_{r \in \I^{+}_i} \rk(Y_r)\leq \sum_{i=1}^n \sigma_{+}(x_i) \dim W(x_i),
\end{equation}
and 
\begin{equation}\label{eqn2-lemma-quivers}
\begin{split}
\rk(T_{V,\sigma}(Y))& =\sum_{j=1}^m \sigma_{-}(y_j) \rk \left( \sum_{i=1}^n \sum_{a \in \ar_{i,j}} V(a) \left(\sum_{r \in \I^+_i} Y_r \right) V(a)^T   \right) \\
& = \sum_{j=1}^m \sigma_{-}(y_j) \rk \left( \sum_{i, a} \sum_{l,r} V(a)(\sqrt{\lambda^{i,r}_l} u^{i,r}_l) \left( V(a) (\sqrt{\lambda^{i,r}_l} u^{i,r}_l) \right)^T  \right) \\
&= \sum_{j=1}^m \sigma_{-}(y_j) \dim \left(\sum_{i=1}^n \sum_{a \in \ar_{i,j}}V(a)(W(x_i))   \right)=\sum_{j=1}^m \sigma_{-}(y_j) \dim W(y_j).
\end{split}
\end{equation}

It now follows from $(\ref{eqn1-lemma-quivers})$ and $(\ref{eqn2-lemma-quivers})$ that
$$
\sigma \cdot \ddim W \geq \rk(Y)-\rk(T_{V, \sigma}(Y)).
$$

\noindent $(3)$ Using the exact same arguments as in the proof of $(iv) \Longrightarrow (iii)$ in Proposition \ref{equiv-cap-0}, we get that
$$
\rk(Y)-\rk(T_{V, \sigma}(Y)) \geq \dim U-\dim \left( \sum_{i=1}^L A_i(U) \right).
$$
This combined with part $(2)$ yields
$$
\sigma \cdot \ddim W \geq c.
$$
\end{proof}

\begin{definition} Let $(V, \sigma)$ be a quiver datum as above. We define the \emph{discrepancy} of $(V,\sigma)$ to be
$$\disc(V, \sigma):=\max \{\sigma \cdot \ddim W \mid W \subseteq V \}.$$ 
\end{definition}

\begin{rmk}
It is clear that $V$ is $\sigma$-semi-table if and only if $\disc(V, \sigma)=0$.
\end{rmk}

Recall that the discrepancy of $\A_{V, \sigma}$ is defined as 
$$\disc(\A_{V, \sigma}):=\max\{c \in \NN \mid \text{there exists a~} c\text{-shrunk subspace for~} \A_{V,\sigma}\}.$$
We mention that $N-\disc(\A_{V, \sigma})$ is the so-called non-commutative rank of the matrix space $\langle \A_{V,\sigma} \rangle \subseteq \RR^{N \times N}$ (see for example \cite{IQS18}).

\begin{corollary}\label{disc-cor} Let $Q$ be a bipartite quiver and $(V, \sigma)$ a quiver datum as above. Then
$$
\disc(V, \sigma)=\disc(\A_{V, \sigma}).
$$
\end{corollary}

\begin{proof} From Lemma \ref{subrep-witness-shrunk-general-lemma}{(3)} we immediately get that $\disc(V, \sigma)\geq \disc(\A_{V, \sigma})$ while the reverse inequality follows from Lemma \ref{subrep-witness-shrunk-general-lemma}{(1)}. 
\end{proof}
 
We are now ready to prove Theorem \ref{main-thm-2}. 
 
\begin{proof}[Proof of Theorem \ref{main-thm-2}] Let $\{x_1, \ldots, x_n\}$ and $\{y_1, \ldots, y_m\}$ be the vertices of $Q$ where $\sigma$ is positive and negative, respectively. Let $Q^{\sigma}$ be the bipartite quiver with partite sets $\{x_1, \ldots, x_n\}$ and $\{y_1, \ldots, y_m\}$, respectively; furthermore, for every oriented path $p$ in $Q$ from $x_i$ to $y_j$, we draw an arrow $a_p$ from $x_i$ to $y_j$ in $Q^{\sigma}$. We denote the restriction of $\dd$ (or $\sigma$) to $Q^{\sigma}$ by the same symbol.

Note that every representation $V$ of $Q$ defines a representation $V^{\sigma}$ of $Q^{\sigma}$ as follows
\begin{itemize}
\item $V^{\sigma}(x_i)=V(x_i)$, $V^{\sigma}(y_j)=V(y_j)$ for all $i \in [n]$, $j \in [m]$, and
\smallskip
\item $V^{\sigma}(a_p)=V(p)$ for every arrow $a_p$ in $Q^{\sigma}$.
\end{itemize}	

Then $V$ (or $V^{\sigma}$) is $\sigma$-semi-stable as a representation of $Q$ (or $Q^{\sigma}$) if and only if the same statement holds over the field of complex numbers (see \cite[Proposition 2.4]{HosSch2017}). On the other hand, it has been proved in \cite[Theorem 10]{ChiKli-Edmonds-2020} that over $\CC$, $V_{\CC}$ is $\sigma$-semi-stable as a representation of $Q$ if and only if $V^{\sigma}_{\CC}$ is $\sigma$-semi-stable as a representation of $Q^{\sigma}$. We thus conclude that

\begin{equation}\label{eqn-semi-Q-Qsigma}
V \text{~is~}\sigma\text{-semi-stable} \Longleftrightarrow V^{\sigma} \text{~is~}\sigma\text{-semi-stable}. 
\end{equation}
 
\noindent \emph{(i)}~ This part of the theorem follows from the second part via $(\ref{eqn-semi-Q-Qsigma})$. 

\bigskip

\noindent \emph{(ii)}~ For this part, we assume that $Q$ is bipartite. According to \cite[Theorem 1.5]{IQS18}, there exists a polynomial time algorithm that constructs a $\disc(\A_{V, \sigma})$-shrunk subspace $U$ for $\A_{V, \sigma}$. Applying Proposition \ref{subrep-witness-shrunk-general-lemma} to the subspace $U$, we can construct a subrepresentation $W$ of $V$ such that
$$
\sigma \cdot \ddim W \geq \disc(\A_{V, \sigma}).
$$
It now follows from Corollary \ref{disc-cor} that
$$
\sigma \cdot \ddim W=\disc(V, \sigma).
$$
In particular, if $\sigma \cdot \ddim W=0$ then $V$ is $\sigma$-semi-stable. Otherwise, $W$ is the desired witness.
\end{proof} 

\begin{remark}
\begin{enumerate}
\item One of the key advantages of the Ivanyos-Qiao-Subrahmanyam's algorithm in \cite{IQS18} over Algorithm \ref{algGlab} is not only that it outputs a $c$-shrunk subspace but it constructs such a shrunk subspace certifying that $c$ is precisely $\disc(\A_{V, \sigma})$. Furthermore, the algorithm is algebraic in nature, working over large enough fields.

\item Keeping in mind the construction of $Q^{\sigma}$ in the proof of Theorem \ref{main-thm-2}, while any subrepresentation of $V$ gives rise to a subrepresentation of $V^{\sigma}$, it is not clear how to extend a given subrepresentation of $V^{\sigma}$ to a subrepresentation of $V$. So, when it comes to producing a certificate for $V$ not being semi-stable in Theorem \ref{main-thm-2}{(1)}, we can only do it in the form of a subrepresentation of $V^{\sigma}$. 

\item As pointed out to us by the anonymous referee and Chi-yu Cheng, it is natural to ask whether an SRSR solution can actually be obtained via a maximally destabilizing 1-parameter subgroup.  We plan to address this question in a future publication.
\end{enumerate}
\end{remark}


\end{document}